\documentclass[10pt]{article}
\usepackage{amsthm,amstext,amssymb,amsmath,graphicx,amsfonts,caption,subcaption,mathrsfs,multirow}
\usepackage{multicol}
\usepackage{xcolor}
\usepackage[sort&compress,round,comma]{natbib}
\usepackage[pagebackref=false,colorlinks,linkcolor=red!50!blue,citecolor=blue]{hyperref}
\theoremstyle{plain}
\newtheorem{theorem}{Theorem}[section]
\newtheorem{lemma}[theorem]{Lemma}
\newtheorem{proposition}[theorem]{Proposition}

\theoremstyle{definition}
\newtheorem{definition}[theorem]{Definition}
\newtheorem{corollary}[theorem]{Corollary}
\newtheorem{example}[theorem]{Example}

\theoremstyle{remark}
\newtheorem{remark}{\sc Remark}
\makeatother
\makeatletter
\def\namedlabel#1#2{\begingroup
   \def\@currentlabel{#2}%
   \label{#1}\endgroup}
\makeatother
\date{}
\title{\bf  $n$-normal residuated lattices}\vspace{.25 in}
\author{ \vspace{.25 in} {\bf  Saeed Rasouli$^1$} and {\bf  Michiro Kondo$^2$}\\
$^1$Persian Gulf University, 75168, Bushehr, Iran \\
{\tt srasouli@pgu.ac.ir }\\\
\\
$^2$ Department of Mathematics,\\
School of System Design and Technology,\\
Tokyo Denki University,\\
Senju 5, Adachi, Tokyo, 120-8551\\
{\tt mkondo@mail.dendai.ac.jp}\\ }
 \begin{document}
 \maketitle
 \begin{abstract}
 The notion of $n$-normal residuated lattice, as a class of residuated lattices in which every prime filter contains at most $n$ minimal prime filters, is introduced and studied. Before that, the notion of $\omega$-filter is introduced and it is observed that the set of $\omega$-filters in a residuated lattice forms a distributive lattice on its own, which includes the set of coannulets as a sublattice. The class of $n$-normal residuated lattices is characterized in terms of their prime filters, minimal prime filters, coannulets and $\omega$-filters.
\end{abstract}
\section{Introduction}\label{sec1}

Distributive pseudo-complemented lattices form an important class of distributive lattices. Garrett Birkhoff asked a question \citep[Problem 70]{bir} inspired by M. H. Stone: ``What is the most general pseudo-complemented distributive lattice in which $x^*\vee x^{**}=1$ identically?" The first solution to this problem belongs to \cite{gras}, who gave the name ``\textit{Stone lattices}" to this class of lattices. They characterized stone lattices as distributive pseudo-complemented lattices in which any pair of incomparable minimal prime ideals is comaximal or equivalently each prime ideal contains a unique minimal prime ideal. Motivated by this characterization, \cite{cor} studied distributive lattices with $0$ in which each prime ideal contains a unique minimal prime ideal under the name ``\textit{normal lattices}". He proved that if $\mathfrak{A}$ is a distributive lattice with $0$, it is normal if and only if for any $x,y\in A$, $x\wedge y=0$ implies $x^{\perp}$ and $y^{\perp}$ are comaximal. Cornish used the ``normal" term in light of \cite{wal}, who proved that the lattice of closed subsets of a $T_1$ space satisfies the above annihilator condition if and only if the space is normal. A complete study on normal lattices can be found in \cite{cor,jo82,za83,pa93}. On the other hand, the results of \cite{gras} generalized by \cite{lee1970}, who considered lattices in which each prime ideal contains at most $n$ minimal prime ideals. \cite{cor74} gave the name ``\textit{$n$-normal lattices}" to this class and presented some of their characterization. The concept of $n$-normality for join-semilattices and posets considered by \cite{nw05} and \cite{hjk10}, respectively.

In this paper, we introduce the notion of $n$-normal residuated lattices and generalize some results of \cite{cor74} and \cite{hjk10} to this class of algebras.

This paper is organized in four sections as follow: In Section \ref{sec2}, some definitions and facts about residuated lattices are recalled and some proposition about prime and minimal prime filters are proved. Also, for a given filter $F$ of a residuated lattice $\mathfrak{A}$, it is recalled that the set of coannihilators belonging to $F$, $\Gamma_{F}(\mathfrak{A})$, forms a complete Boolean algebra on its own, and the set of coannulets belonging to $F$, $\gamma_{F}(\mathfrak{A})$, is a sublattice of $\Gamma_{F}(\mathfrak{A})$. In Section \ref{sec3}, notions of $\omega$-filters and divisor filters, as an especial subclass of $\omega$-filters in a residuated lattice, are introduced and some properties of them are studied. For a given residuated lattice $\mathfrak{A}$ and a filter $F$ of $\mathfrak{A}$ it is shown that the set of $\omega$-filters belonging to $F$, $\Omega_{F}(\mathfrak{A})$, forms a distributive lattice on its own, and $\gamma_{F}(\mathfrak{A})$ is a sublattice of $\Omega_{F}(\mathfrak{A})$. Also, it is shown that for a prime filter $P$ containing $F$ the set of $F$-divisors of $P$, $D_{F}(P)$, is the intersection of $F$-minimal prime filters of $\mathfrak{A}$ and $P$ is $F$-minimal if and only if $P=D_{F}(P)$. In Section \ref{sec4}, the notion of $n$-normal residuated lattice is introduced and characterized by applying of prime filters and minimal prime filters. Normal residuated lattices are characterized as those one their lattice of $\omega$-filters are a sublattice of their lattice of filters. Finally, it is proved that in a normal residuated lattice the greatest $\omega$-filter contained in a filter exists.
\section{Residuated lattices}\label{sec2}

In this section, we recall some definitions, properties and results relative to residuated lattices, which will be used
in the following. The results in the this section are original, excepting those that we cite from other papers.

An algebra $\mathfrak{A}=(A;\vee,\wedge,\odot,\rightarrow,0,1)$ is called a \textit{residuated lattice} if $\ell(\mathfrak{A})=(A;\vee,\wedge,0,1)$ is a bounded lattice, $(A;\odot,1)$ is a commutative monoid and $(\odot,\rightarrow)$ is an adjoint pair. A residuated lattice $\mathfrak{A}$ is called a \textit{MTL algebra} if satisfying the \textit{pre-linearity condition} (denoted by \ref{prel}):
\begin{enumerate}
\item [$(prel)$ \namedlabel{prel}{$(prel)$}] $(x\rightarrow y)\vee(y\rightarrow x)=1$, for all $x,y\in A$.
\end{enumerate}

In a residuated lattice $\mathfrak{A}$, for any $a\in A$, we put $\neg a:=a\rightarrow 0$. It is well-known that the class of residuated lattices is equational \citep{idz}, and so it forms a variety. The properties of residuated lattices were presented in \cite{gal}. For a survey of residuated lattices we refer to \cite{jip}.
\begin{remark}\label{resproposition}\citep[Proposition 2.2]{jip}
Let $\mathfrak{A}$ be a residuated lattice. The following conditions are satisfied for any $x,y,z\in A$:
\begin{enumerate}
  \item [$r_{1}$ \namedlabel{res1}{$r_{1}$}] $x\odot (y\vee z)=(x\odot y)\vee (x\odot z)$;
  \item [$r_{2}$ \namedlabel{res2}{$r_{2}$}] $x\vee (y\odot z)\geq (x\vee y)\odot (x\vee z)$.
  \end{enumerate}
\end{remark}
\begin{example}\label{rex2}
Let $A_6=\{0,a,b,c,d,1\}$ be a lattice whose Hasse diagram is below (see Figure \ref{graph6}).  Define $\odot$ and $\rightarrow$ on $A_7$ as follows:
\begin{eqnarray*}
\begin{array}{l|llllll}
  \odot & 0 & a & b & c & d &  1   \\ \hline
    0   & 0 & 0 & 0 & 0 & 0 &  0  \\
    a   & 0 & a & a & 0 & a &  a  \\
    b   & 0 & a & a & 0 & a &  b  \\
    c   & 0 & 0 & 0 & c & c &  c  \\
    d   & 0 & a & a & c & d &  d  \\
    1   & 0 & a & b & c & d &  1
  \end{array}& \hspace{1cm} &
  \begin{array}{l|llllll}
  \rightarrow & 0 & a & b & c & d &  1   \\ \hline
    0         & 1 & 1 & 1 & 1 & 1 &  1  \\
    a         & c & 1 & 1 & c & 1 &  1  \\
    b         & c & d & 1 & c & 1 &  1  \\
    c         & b & b & b & 1 & 1 &  1  \\
    d         & 0 & b & b & c & 1 &  1  \\
    1         & 0 & a & b & c & d &  1
  \end{array}
\end{eqnarray*}
\begin{figure}[h]
\centering
\includegraphics[scale=.15]{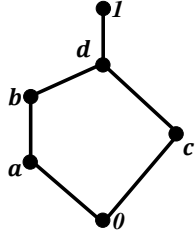}
\caption{The Hasse diagram of $\mathfrak{A}_6$.}
\label{graph6}
\end{figure}
Routine calculation shows that  $\mathfrak{A}_6=(A_6;\vee,\wedge,\odot,\rightarrow,0,1)$ is a residuated lattice.
\end{example}

Let $\mathfrak{A}$ be a residuated lattice. A non-void subset $F$ of $A$ is called a \textit{filter} of $\mathfrak{A}$ if $x,y\in F$ implies $x\odot y\in F$ and $x\vee y\in F$ for any $x\in F$ and $y\in A$. The set of filters of $\mathfrak{A}$ is denoted by $\mathscr{F}(\mathfrak{A})$. A filter $F$ of $\mathfrak{A}$ is called \textit{proper} if $F\neq A$. Clearly, $F$ is a proper filter if and only if $0\notin F$. For any subset $X$ of $A$ the \textit{filter of $\mathfrak{A}$ generated by $X$} is denoted by $\mathscr{F}(X)$. For each $x\in A$, the filter generated by $\{x\}$ is denoted by $\mathscr{F}(x)$ and called \textit{principal filter}. The set of principal filters is denoted by $\mathscr{PF}(\mathfrak{A})$. Let $\mathcal{F}$ be a collection of filters of $\mathfrak{A}$. Set $\veebar \mathcal{F}=\mathscr{F}(\cup \mathcal{F})$. It is well-known that $(\mathscr{F}(\mathfrak{A});\cap,\veebar,\textbf{1},A)$ is a frame and so it is a complete Heyting algebra.
\begin{example}\label{fex2}
Consider the residuated lattice $\mathfrak{A}_6$ from Example \ref{rex2}. Then $\mathscr{F}(\mathfrak{A}_6)=\{F_1=\{1\},F_2=\{d,1\},F_3=\{a,b,d,1\},F_4=\{c,d,1\},F_5=A_6\}$.
\end{example}

The following remark has a routine verification.
\begin{remark}\label{genfilprop}
Let $\mathfrak{A}$ be a residuated lattice and $F$ be a filter of $\mathfrak{A}$. The following assertions hold for any $x,y\in A$:
\begin{enumerate}
  \item  [$(1)$ \namedlabel{genfilprop1}{$(1)$}] $\mathscr{F}(F,x):=F\veebar \mathscr{F}(x)=\{a\in A|f\odot x^n\leq a,~f\in F\}$;
  \item  [$(2)$ \namedlabel{genfilprop2}{$(2)$}] $x\leq y$ implies $\mathscr{F}(F,y)\subseteq \mathscr{F}(F,x)$.
  \item  [$(3)$ \namedlabel{genfilprop3}{$(3)$}] $\mathscr{F}(F,x)\cap \mathscr{F}(F,y)=\mathscr{F}(F,x\vee y)$;
  \item  [$(4)$ \namedlabel{genfilprop4}{$(4)$}] $\mathscr{F}(F,x)\veebar \mathscr{F}(F,y)=\mathscr{F}(F,x\odot y)$;
  \item  [$(5)$ \namedlabel{genfilprop5}{$(5)$}] $\mathscr{PF}(\mathfrak{A})$ is a sublattice of $\mathscr{F}(\mathfrak{A})$.
\end{enumerate}
\end{remark}

A proper filter of a residuated lattice $\mathfrak{A}$ is called \textit{maximal} if it is a maximal element in the set of all proper filters. The set of all maximal filters of $\mathfrak{A}$ is denoted by $Max(\mathfrak{A})$. A proper filter $P$ of $\mathfrak{A}$ is called \textit{prime}, if for any $x,y\in A$, $x\vee y\in P$ implies $x\in P$ or $y\in P$. The set of all prime filters of $\mathfrak{A}$ is denoted by $Spec(\mathfrak{A})$. Since $\mathscr{F}(\mathfrak{A})$ is a distributive lattice, so $Max(\mathfrak{A})\subseteq Spec(\mathfrak{A})$. By Zorn's lemma follows that any proper filter is contained in a maximal filter and so in a prime filter.

A non-empty subset $\mathscr{C}$ of $\mathfrak{A}$ is called \textit{$\vee$-closed} if it is closed under the join operation, i.e $x,y\in \mathscr{C}$ implies $x\vee y\in \mathscr{C}$.
\begin{remark}\label{primclos}
It is obvious that a filter $P$ is prime if and only if $P^{c}$ is $\vee$-closed. Also, if $\mathscr{P}\subseteq Spec(\mathfrak{A})$, then $(\cup \mathscr{P})^{c}$ is a $\vee$-closed subset of $\mathfrak{A}$.
\end{remark}

The following result is an easy consequence of Zorn's lemma.
\begin{lemma}\label{0prfiltth}
If $\mathscr{C}$ is a $\vee$-closed subset of $\mathfrak{A}$ which does not meet the filter $F$, then $\mathscr{C}$ is contained in a $\vee$-closed subset $\textsf{C}$ which is maximal with respect to the property of not meeting $F$.
\end{lemma}

The following important result is proved for pseudo-BL algebras \cite[Theorem 4.28]{din0}; however, it can be proved without difficulty in all residuated lattices.
\begin{theorem}\label{prfilth}
If $\mathscr{C}$ is a $\vee$-closed subset of $\mathfrak{A}$ which does not meet the filter $F$, then $F$ is contained in a filter $P$ which is maximal with respect to the property of not meeting $\mathscr{C}$; furthermore $P$ is prime.
\end{theorem}
\begin{corollary}\label{intprimfilt}
Let $F$ be a filter of a residuated lattice $\mathfrak{A}$ and $X$ be a subset of $A$. The following assertions hold:
\begin{enumerate}
\item  [(1) \namedlabel{intprimfilt1}{(1)}]  If $X\nsubseteq F$, there exists a prime filter $P$ such that $F\subseteq P$ and $X\nsubseteq P$;
\item  [(2) \namedlabel{intprimfilt2}{(2)}] $\mathscr{F}(X)=\bigcap \{P\in Spec(\mathfrak{A})|X\subseteq P\}$.
\end{enumerate}
\end{corollary}
\begin{proof}
\begin{enumerate}
  \item [\ref{intprimfilt1}:] Let $x\in X-F$. By taking $\mathscr{C}=\{x\}$ it follows by Theorem \ref{prfilth}.
  \item [\ref{intprimfilt2}:] Set $\sigma_{X}=\{P\in Spec(\mathfrak{A})|X\subseteq P\}$. Obviously, we have $\mathscr{F}(X)\subseteq \bigcap \sigma_{X}$. Now let $a\notin \mathscr{F}(X)$. By \ref{intprimfilt1} follows that there exits a prime filter $P$ containing $\mathscr{F}(X)$ such that $a\notin P$. It shows that $a\notin \bigcap \sigma_{X}$.
\end{enumerate}
\end{proof}

Let $\mathfrak{A}$ be a residuated lattice and $X$ be a subset of $A$. A prime filter $P$ is called a \textit{minimal prime filter belonging to $X$} or $X$-\textit{minimal prime filter} if $P$ is a minimal element in the set of prime filters containing $X$. The set of $X$-minimal prime filters of $\mathfrak{A}$ is denoted by $Min_{X}(\mathfrak{A})$. A prime filter $P$ is called a \textit{minimal prime} if $P\in Min_{\{1\}}(\mathfrak{A})$. The set of minimal prime filters of $\mathfrak{A}$ is denoted by $Min(\mathfrak{A})$.

In following we give an important characterization for minimal prime filters.
\begin{theorem}\label{1mineq}
Let $\mathfrak{A}$ be a residuated lattice and $F$ be a filter of $\mathfrak{A}$. A subset $P$ of $A$ is an $F$-minimal prime filter if and only if $P^{c}$ is a $\vee$-closed subset of $\mathfrak{A}$ which it is maximal with respect to the property of not meeting $F$.
\end{theorem}
\begin{proof}
Let $P$ be a subset of $A$ such that $P^{c}$ is a $\vee$-closed subset of $\mathfrak{A}$
which is maximal w.r.t the property of not meeting $F$. By Proposition \ref{prfilth} there exists a prime filter $Q$ such that $Q$ not meeting $P^{c}$ and so $Q\subseteq P$. By \textsc{Remark} \ref{primclos}, $Q^{c}$ is a $\vee$-closed subset of $\mathfrak{A}$ and by hypothesis we have $P^{c}\subseteq Q^{c}$ and $Q^{c}\cap F=\emptyset$. So by maximality of $P^{c}$ we deduce that $P^{c}=Q^{c}$ and it means that $P=Q$. It shows that $P$ is a prime filter and moreover it shows that $P$ is an $F$-minimal prime filter.

Conversely, let $P$ be an $F$-minimal prime filter of $\mathfrak{A}$. By \textsc{Remark} \ref{primclos}, $P^{c}$ is a $\vee$-closed subset of $\mathfrak{A}$ such that $P^{c}\cap F=\emptyset$. By using Lemma \ref{0prfiltth} we can obtain a $\vee$-closed subset $\mathscr{C}$ of $\mathfrak{A}$ such that it is maximal with respect to the property of not meeting $F$. By case just proved, $\mathscr{C}'$ is an $F$-minimal prime filter such that $\mathscr{C}'\cap P^{c}=\emptyset$ and it implies $\mathscr{C}'\subseteq P$. By hypothesis $\mathscr{C}=P^{c}$ and it shows that $P^{c}$ is a $\vee$-closed subset of $\mathfrak{A}$ such that it is maximal with respect to the property of not meeting $F$.
\end{proof}
\begin{corollary}\label{primeminimal}
Let $\mathfrak{A}$ be a residuated lattice, $X$ be a subset of $A$ and $P$ be a prime filter containing $X$. Then there exists an $X$-minimal prime filter contained in $P$.
\end{corollary}
\begin{proof}
By \textsc{Remark} \ref{primclos}, $P^{c}$ is a $\vee$-closed subset of $\mathfrak{A}$ such that $P^{c}\cap \mathscr{F}(X)=\emptyset$. By using Lemma \ref{0prfiltth} we can obtain a $\vee$-closed subset $\mathscr{C}$ of $\mathfrak{A}$ containing $P^{c}$ such that it is maximal with respect to the property of not meeting $\mathscr{F}(X)$. By Theorem \ref{1mineq}, $\mathscr{C}'$ is an $\mathscr{F}(X)$-minimal prime filter which it is contained in $P$.
\end{proof}

The following corollary should be compared with Corollary \ref{intprimfilt}.
\begin{corollary}\label{mininters}
Let $F$ be a filter of a residuated lattice $\mathfrak{A}$ and $X$ be a subset of $A$. The following assertions hold:
\begin{enumerate}
\item  [(1) \namedlabel{mininters1}{(1)}]  If $X\nsubseteq F$, there exists an $F$-minimal prime filter $\mathfrak{m}$ such that $X\nsubseteq \mathfrak{m}$;
\item  [(2) \namedlabel{mininters2}{(2)}] $\mathscr{F}(X)=\bigcap Min_{X}(\mathfrak{A})$.
\end{enumerate}
\end{corollary}
\begin{proof}
\begin{enumerate}
  \item [\ref{mininters1}:] It is a direct consequence of Corollary \ref{intprimfilt}\ref{intprimfilt1} and Corollary \ref{primeminimal}.
  \item [\ref{mininters2}:] Set $\sigma_{X}=\{P\in Spec(\mathfrak{A})|X\subseteq P\}$. By Corollary \ref{intprimfilt}\ref{intprimfilt2}, it is sufficient to show that $\bigcap Min_{X}(\mathfrak{A})=\bigcap \sigma_{X}$. It is obvious that $\bigcap \sigma_{X}\subseteq \bigcap Min_{X}(\mathfrak{A})$. Otherwise, let $a\in \bigcap Min_{X}(\mathfrak{A})$ and $P$ be an arbitrary element of $\sigma_{X}$. By Corollary \ref{primeminimal} there exists an $X$-minimal prime filter $\mathfrak{m}$ contained in $P$. Hence, $a\in \mathfrak{m}\subseteq P$ and it states that $\bigcap Min_{X}(\mathfrak{A})\subseteq \bigcap \sigma_{X}$.
\end{enumerate}
\end{proof}

Let $\mathfrak{A}$ be a residuated lattice and $F$ be a filter of $\mathfrak{A}$. Recalling that \citep{rascoan} for any subset $X$ of $A$ the \textit{coannihilator of $X$ belonging to $F$} (or, $F$-\textit{coannihilator of $X$}) is denoted by $(F:X)$ and defined as follow:
\[(F:X)=\{a\in A|x\vee a\in F,\forall x\in X\}.\]
If $X=\{x\}$, we write $(F:x)$ instead of $(F:X)$ and in case $F=\{1\}$, we write $X^{\perp}$ instead of $(F:X)$.
\begin{example}\label{ppfex0}
Consider the residuated lattice $\mathfrak{A}_6$ from Example \ref{rex2}. With notations of Example \ref{fex2} we have $(F_4:0)=F_4$, $(F_4:a)=F_4$, $(F_4:b)=F_4$, $(F_4:c)=F_5$, $(F_4:d)=F_5$ and $(F_4:1)=F_5$.
\end{example}
\begin{remark}
Let $\mathfrak{A}$ be a residuated lattice, $F$ be a filter of $\mathfrak{A}$ and $X$ be a subset of $A$. One can see that $(F:X)$ is the relative pseudo-complement of $\mathscr{F}(X)$ with respect to $F$ in the lattice $\mathscr{F}(\mathfrak{A})$.
\end{remark}

In the following proposition we recall some properties of coannihilators.
\begin{proposition}\citep[Proposition 3.1]{rascoan}\label{1fxpro}
Let $\mathfrak{A}$ be a residuated lattice and $F$ be a filter of $\mathfrak{A}$. The following assertions hold for any $X,Y\subseteq A$:
\begin{enumerate}
\item  [$(1)$ \namedlabel{1fxpro1}{$(1)$}] $X\subseteq (F:Y)$ implies $Y\subseteq (F:X)$;
\item  [$(2)$ \namedlabel{1fxpro2}{$(2)$}] $(F:X)=A$ if and only if $X\subseteq F$;
\end{enumerate}
\end{proposition}

Let $\mathfrak{A}$ be a residuated lattice and $F$ be a filter of $\mathfrak{A}$. Set $\Gamma_{F}(\mathfrak{A})=\{(F:X)|X\subseteq A\}$. The elements of $\Gamma_{F}(\mathfrak{A})$ are called \textit{$F$-coannihilators} of $\mathfrak{A}$. We recall that $(\Gamma_{F}(\mathfrak{A});\cap,\vee^{\Gamma_{F}},F,A)$ is a complete Boolean lattice, where for any $\mathscr{G}\subseteq \Gamma_{F}(\mathfrak{A})$ we have $\vee^{\Gamma_{F}} \mathscr{F}=(F:(F:\cup \mathscr{G}))$ \citep[Proposition 3.13]{rascoan}.
\begin{proposition}\citep[Proposition 3.15]{rascoan}\label{4fxpro}
Let $\mathfrak{A}$ be a residuated lattice and $F$ be a filter of $\mathfrak{A}$. The following assertions hold for any $x,y\in A$:
\begin{enumerate}
\item  [$(1)$ \namedlabel{4fxpro1}{$(1)$}] $x\leq y$ implies $(F:x)\subseteq (F:y)$;
\item  [$(2)$ \namedlabel{4fxpro2}{$(2)$}] $(F:x)\cap (F:y)=(F:x\odot y)$;
\item  [$(3)$ \namedlabel{4fxpro3}{$(3)$}] $(F:(F:x))\cap (F:(F:y))=(F:(F:x\vee y))$;
\item  [$(4)$ \namedlabel{4fxpro4}{$(4)$}] $(F:x)\veebar (F:y)\subseteq (F:x)\vee^{\Gamma_F} (F:y)=(F:x\vee y)$.
\end{enumerate}
\end{proposition}

Let $\mathfrak{A}$ be a residuated lattice. We set $\gamma_{F}(\mathfrak{A})=\{(F:x)|x\in A\}$. The elements of $\gamma_{F}(\mathfrak{A})$ are called \textit{$F$-coannulets} of $\mathfrak{A}$. Applying Proposition \ref{4fxpro1}, it follows that $\gamma_{F}(\mathfrak{A})$ is a sublattice of $\Gamma_{F}(\mathfrak{A})$ \citep[Theorem 3.16]{rascoan}.
\section{$\omega$-filters}\label{sec3}

In this section we introduce and investigate the notion of $\omega$-filters in a residuated lattice.
\begin{definition}\label{geomegafilter}
Let $\mathfrak{A}$ be a residuated lattice and $F$  be a filter of $\mathfrak{A}$. For any subset $X$ of $A$ we set:
\[\omega_{F}(X)=\{a\in A|x\vee a\in F,\exists x\in X\}.\]
In the following, $\omega_{\{1\}}(X)$ shall be denoted by $\omega(X)$.
\end{definition}
\begin{remark}
The notions of $O(P)$ for a prime ideal $P$ and its dual, $\omega(P)$ for a prime filter $P$, in a distributive lattice with $0$ are introduced in \cite{cor}, where it is shown that $O(P)$ is the intersection of all minimal prime ideals contained in $P$ \citep[Proposition 2.2]{cor}. In \cite{cor1}, these ideals are employed as examples of $\alpha$-ideals. In \cite{cor2}, the notion of $O$-ideals in bounded distributive lattices are introduced and their properties by means of congruence relations are applied for obtaining a sheaf representation (by a ``sheaf representation" of a bounded distributive lattice $\mathfrak{A}$ the author mean a sheaf representation whose base space is $Space(\mathfrak{A})$ and whose stalks are the quotients $\mathfrak{A}/O(P)$, where P is a prime ideal).  L. Leu\c{s}tean \citep{leus} introduced the notion of $O$-filters in BL-algebras as the dual of $o$-ideals studied by Cornish. $O$-ideals are the lattice version of the following ideals in rings: if $R$ is a ring, then $O(P)=\{a\in R|~as=0,~ \textrm{for~some}~s\in R\setminus P\}$, where $P$ is a prime ideal of $R$. $O$-ideals are used for obtaining sheaf representations of different classes of rings \citep{bir,bir1,hof}.
\end{remark}

Let $\mathfrak{A}$ be a residuated lattice and $F$ be a filter of $\mathfrak{A}$. A subset $X$ of $\mathfrak{A}$ shall be called \textit{$F$-dense} if $(F:X)=F$. The set of all $F$-dense elements of $\mathfrak{A}$ shall be denoted by $\mathfrak{D}_{F}(\mathfrak{A})$. By Proposition \ref{4fxpro}(\ref{4fxpro1} and \ref{4fxpro4}) follows that $\mathfrak{D}_{F}(\mathfrak{A})$ is an ideal of $\ell(\mathfrak{A})$.

\begin{proposition}\label{1propdiv}
Let $\mathfrak{A}$ be a residuated lattice, $F,G$  be filters and $X,Y$ be subsets of $A$. The following assertions hold:
\begin{enumerate}
\item  [$(1)$ \namedlabel{1propdiv1}{$(1)$}] $\omega_{F}(X)=\cup_{x\in X}(F:x)$;
\item  [$(2)$ \namedlabel{1propdiv2}{$(2)$}] $\omega_{F}(X)=\{a\in A|(F:a)\cap X\neq\emptyset\}$;
\item  [$(3)$ \namedlabel{1propdiv3}{$(3)$}] $F\subseteq \omega_{F}(X)$;
\item  [$(4)$ \namedlabel{1propdiv4}{$(4)$}] $X\subseteq Y$ implies $\omega_{F}(X)\subseteq \omega_{F}(Y)$;
\item  [$(5)$ \namedlabel{1propdiv5}{$(5)$}] $F\subseteq G$ implies $\omega_{F}(X)\subseteq \omega_{G}(X)$;
\item  [$(6)$ \namedlabel{1propdiv6}{$(6)$}] $\omega_{F}(X)=A$ if and only if $F\cap X\neq \emptyset$;
\item  [$(7)$ \namedlabel{1propdiv7}{$(7)$}] $\omega_{F}(X)=F$ if and only if $X\subseteq \mathfrak{D}_{F}(\mathfrak{A})$.
\end{enumerate}
\end{proposition}
\begin{proof}
We only prove the cases \ref{1propdiv6} and \ref{1propdiv7}, because the other cases can be proved in a routine way.
\begin{enumerate}
  \item [\ref{1propdiv6}] If $\omega_{F}(X)=A$, then $0\in \omega_{F}(X)$ and it implies that $0\in (F:x)$ for some $x\in X$. So $x\in (F:0)=F$ and it means that $F\cap X\neq \emptyset$. Otherwise, if $x\in F\cap X$, then we have $(F:x)=A$ and it follows that $\omega_{F}(X)=A$.
  \item [\ref{1propdiv7}] Let $\omega_{F}(X)=F$ and $x\in X$. So we have $F\subseteq (F:x)\subseteq \omega_{F}(X)=F$ and it states that $x\in \mathfrak{D}_{F}(\mathfrak{A})$. Conversely, $X\subseteq \mathfrak{D}_{F}(\mathfrak{A})$ states that $(F:x)=F$ for any $x\in X$ and it follows that $\omega_{F}(X)=F$.
\end{enumerate}
\end{proof}
\begin{proposition}\label{2propdiv}
Let $\mathfrak{A}$ be a residuated lattice and $\mathscr{C}$ be a $\vee$-closed subset of $\mathfrak{A}$. Then $\omega_{F}(\mathscr{C})$ is a filter.
\end{proposition}
\begin{proof}
By Proposition \ref{1propdiv}\ref{1propdiv3} we have $1\in \omega_{F}(\mathscr{C})$. If $a\leq b$ and $a\in \omega_{F}(\mathscr{C})$, then $a\in (F:c)$ for some $c\in \mathscr{C}$ and so $c\in (F:a)$. By \textsc{Remark} \ref{4fxpro}\ref{4fxpro2} follows that $c\in (F:b)$ and it shows that $b\in (F:c)\subseteq \omega_{F}(\mathscr{C})$. If $a,b\in \omega_{F}(\mathscr{C})$, then for some $c_a,c_b\in \mathscr{C}$ we have $a\in (F:c_a)$ and $b\in (F:c_b)$. Hence we have $c_a\in (F:a)$ and $c_b\in (F:b)$. By \textsc{Remark} \ref{4fxpro}\ref{4fxpro3} we have $c_a\vee c_b\in (F:a\odot b)$ and it follows that $a\odot b\in (F:c_a\vee c_b)\subseteq \omega_{F}(\mathscr{C})$.
\end{proof}

The next proposition should be compared with Proposition \ref{1propdiv}\ref{1propdiv6}.
\begin{proposition}\label{profilt}
Let $\mathfrak{A}$ be a residuated lattice and $\mathscr{C}$ be a $\vee$-closed subset of $\mathfrak{A}$. The following assertions are equivalent:
\begin{enumerate}
  \item [$(1)$ \namedlabel{profilt1}{$(1)$}] $F\cap \mathscr{C}=\emptyset$;
  \item [$(2)$ \namedlabel{profilt2}{$(2)$}] $\omega_{F}(\mathscr{C})$ is proper;
  \item [$(3)$ \namedlabel{profilt3}{$(3)$}] $\omega_{F}(\mathscr{C})\cap \mathscr{C}=\emptyset$.
\end{enumerate}
\end{proposition}
\begin{proof}
By Proposition \ref{1propdiv}\ref{1propdiv6} follows that \ref{profilt1} and \ref{profilt2} are equivalent.
\item [] \ref{profilt2}$\Rightarrow$ \ref{profilt3}: If $\omega_{F}(\mathscr{C})\cap \mathscr{C}\neq\emptyset$, then $F\cap \mathscr{C}\neq\emptyset$ and so by Proposition \ref{1propdiv}\ref{1propdiv6} follows that $\omega_{F}(\mathscr{C})=A$; a contradiction.
\item [] \ref{profilt3}$\Rightarrow$ \ref{profilt2}: It is evident.
\end{proof}

We recall that a subset $I$ of a lattice $\mathfrak{A}$ is called an \textit{ideal} if $I$ is a $\vee$-closed subset of $\mathfrak{A}$ and $x\leq y$ implies $x\in I$ for any $x\in A$ and $y\in I$.  The set of all ideals of a lattice $\mathfrak{A}$ is denoted by $\mathscr{I}(\mathfrak{A})$. For any subset $X$ of $A$ the \textit{ideal of $\mathfrak{A}$ generated by $X$} is denoted by $\mathscr{I}(X)$ and $\mathscr{I}(\{x\})$ is denoted by $\mathscr{I}(x)$. The following remark has a routine verification.
\begin{remark}\label{idealremark}
Let $\mathfrak{A}$ be a residuated lattice. The following assertions hold for any $x,y\in A$:
\begin{enumerate}
  \item [$(1)$ \namedlabel{idealremark1}{$(1)$}] $(\mathscr{I}(\ell(\mathfrak{A}));\cap,\curlyvee)$ is a frame where $\curlyvee(\mathcal{I})=\mathscr{I}(\cup \mathcal{I})$ for any $\mathcal{I}\subseteq \mathscr{I}(\ell(\mathfrak{A}))$;
  \item [$(2)$ \namedlabel{idealremark2}{$(2)$}] $\mathscr{I}(x)=\{a\in A|a\leq x\}$;
  \item [$(3)$ \namedlabel{idealremark3}{$(3)$}] $\mathscr{I}(x)\cap \mathscr{I}(y)=\mathscr{I}(x\wedge y)$;
  \item [$(4)$ \namedlabel{idealremark4}{$(4)$}] $\mathscr{I}(x)\curlyvee \mathscr{I}(y)=\mathscr{I}(x\vee y)$.
\end{enumerate}
\end{remark}
\begin{definition}\label{omefildef}
A filter $H$ of $\mathfrak{A}$ is called an $\omega$-filter belonging to $F$ (or $\omega_F$-filter) if $H=\omega_{F}(I_{H})$ for some ideal $I_{H}$. The set of all $\omega_F$-filters is denoted by $\Omega_{F}(\mathfrak{A})$. It is obvious that $F,A\in \Omega_{F}(\mathfrak{A})$. In the sequel $\Omega_{\textbf{1}}(\mathfrak{A})$ simply is denoted by $\Omega(\mathfrak{A})$ and its elements are called $\omega$-filters.
\end{definition}

\begin{proposition}\label{dislattoff}
Let $\mathfrak{A}$ be a residuated lattice and $F$ be a filter of $\mathfrak{A}$. Then $(\Omega_{F}(\mathfrak{A});\cap,\vee^{\omega_{F}},F,A)$ is a bounded distributive lattice where $G\vee^{\omega_{F}} H=\omega_{F}(I_{G}\curlyvee I_{H})$ for any $G,H\in \Omega_{F}(\mathfrak{A})$.
\end{proposition}
\begin{proof}
Let $G,H\in \Omega_{F}(\mathfrak{A})$. In a routine way we can show that $G\cap H=\omega_{F}(I_{G}\cap I_{H})$ and $G\vee^{\omega_{F}} H$ is the supremum of $G$ and $H$. Following by \textsc{Remark} \ref{idealremark}\ref{idealremark1} obviously $\Omega_{F}(\mathfrak{A})$ is a distributive lattice.
%
\end{proof}
\begin{lemma}\label{gammaomega}
Let $\mathfrak{A}$ be a residuated lattice and $F$ be a filter of $\mathfrak{A}$. Then $\gamma_{F}(\mathfrak{A})$ is a subset of $\Omega_{F}(\mathfrak{A})$.
\end{lemma}
\begin{proof}
It follows by Proposition \ref{4fxpro}\ref{4fxpro1} and Proposition \ref{1propdiv}\ref{1propdiv1}.
\end{proof}
\begin{proposition}\label{30fxpro}
Let $\mathfrak{A}$ be a residuated lattice and $F$ be a filter of $\mathfrak{A}$. Then $\gamma_{F}(\mathfrak{A})$ is a bounded sublattice of $\Omega_{F}(\mathfrak{A})$.
\end{proposition}
\begin{proof}
By Lemma \ref{gammaomega} follows that $\gamma_{F}(\mathfrak{A})$ is a subset of $\Omega_{F}(\mathfrak{A})$. Also, for any $x,y\in A$ we have the following sequence of formulas:
\[
\begin{array}{ll}
   (F:x)\vee^{\omega_{F}} (F:y)&=\omega_{F}(\mathscr{I}(x))\vee^{\omega_F} \omega_{F}(\mathscr{I}(y))  \\
   &=\omega_{F}(\mathscr{I}(x)\curlyvee \mathscr{I}(y))]  \\
   &=\omega_{F}(\mathscr{I}(x\vee y)) \\
   &=(F:x\vee y).
\end{array}
\]
\end{proof}
\begin{corollary}\label{lastcor}
Let $\mathfrak{A}$ be a residuated lattice and $F$ be a filter of $\mathfrak{A}$. If $x\vee y\in F$, then $(F:x)\vee^{\omega_{F}} (F:y)=A$.
\end{corollary}
\begin{proof}
It is an immediate consequence of Proposition \ref{1fxpro}\ref{1fxpro2} and Proposition \ref{30fxpro}.
\end{proof}

Now, we introduce the notion of divisor filters in a residuated lattice as a special kind of $\omega$-filters, which are important tools in studying of minimal prime filters.

\begin{definition}
Let $H$ be a proper filter of a residuated lattice $\mathfrak{A}$. We set $D_{F}(H)=\omega_{F}(H^{c})$ and call its elements \textit{$F$-divisors of $H$}. $D_{\{1\}}(H)$ is denoted by $D(H)$ and its elements are called \textit{unit divisors of $H$}. Unit divisors of $\{1\}$ simply are called \textit{unit divisors}.
\end{definition}
\begin{proposition}\label{lemdivi}
 Let $\mathfrak{A}$ be a residuated lattice, $F$ be a filter and $H$ be a proper filter of $\mathfrak{A}$. The following assertions hold:
\begin{enumerate}
  \item  [$(1)$ \namedlabel{lemdivi1}{$(1)$}] $F\subseteq D_{F}(H)=\{a\in A|(F:a)\nsubseteq H\}=\cup_{x\notin H}(F:x)$;
  \item  [$(2)$ \namedlabel{lemdivi2}{$(2)$}] $D_{F}(H)=A$ if and only if $F\nsubseteq H$.
\end{enumerate}
 \end{proposition}
\begin{proof}
It is straightforward by Proposition \ref{1propdiv}.
\end{proof}

\begin{proposition}\label{divisorfilt}
Let $\mathfrak{A}$ be a residuated lattice. For any prime filter $P$ of $\mathfrak{A}$ we have the following assertions:
\begin{enumerate}
  \item  [$(1)$ \namedlabel{divisorfilt1}{$(1)$}] $D_{F}(P)$ is an $\omega$-filter of $\mathfrak{A}$;
  \item  [$(2)$ \namedlabel{divisorfilt2}{$(2)$}] if $P$ contains $F$, then $D_{F}(P)\subseteq P$.

\end{enumerate}
\end{proposition}
\begin{proof}
\begin{enumerate}
\item [\ref{divisorfilt1}:] It follows by \textsc{Remark} \ref{primclos} and Proposition \ref{2propdiv}.
\item [\ref{divisorfilt2}:] Let $P$ contains $F$. Since $(F:a)\subseteq P$ for any $a\notin P$ so it follows by Proposition \ref{lemdivi}\ref{lemdivi1}.
\end{enumerate}
\end{proof}

Let $\mathfrak{A}$ be a residuated lattice and $F$ be a filter of $\mathfrak{A}$. A filter $G$ of $\mathfrak{A}$ is called an $F$-divisor filter if $G=D_{F}(P)$ for some prime filter $P$.
\begin{proposition}\label{minlem1}
Let $\mathfrak{A}$ be a residuated lattice, $F$ be a filter and $\mathfrak{m}$ be an $F$-minimal prime filter. The following assertions hold:
\begin{enumerate}
  \item  [$(1)$ \namedlabel{minlem11}{$(1)$}] $\mathfrak{m}=D_{F}(\mathfrak{m})$;
  \item  [$(2)$ \namedlabel{minlem12}{$(2)$}] $\mathfrak{m}\subseteq D_{F}(F)$.
\end{enumerate}
\end{proposition}
\begin{proof} \item [\ref{minlem11}:]  Let $x\in \mathfrak{m}$. It is easy to check that $\mathscr{C}=(x\vee \mathfrak{m}^{c})\cup \mathfrak{m}^{c}$ is a $\vee$-closed subset of $\mathfrak{A}$. By Proposition \ref{1mineq} we obtain that $(x\vee \mathfrak{m}^{c})\cap F=\mathscr{C}\cap F\neq \emptyset$. Assume that $a\in (x\vee \mathfrak{m}^{c})\cap F$. So there exists $y\in \mathfrak{m}^{c}$ such that $x\vee y=a\in F$. The converse inclusion is evident by Proposition \ref{divisorfilt}\ref{divisorfilt2}.
\item [\ref{minlem12}:] It is an immediate consequence of Proposition \ref{1propdiv}\ref{1propdiv4} and \ref{minlem11}.
\end{proof}
\begin{remark}
Applying Proposition \ref{minlem1}\ref{minlem12}, it follows that any element of a minimal prime filter in a residuated lattice is a unit divisor.
\end{remark}

The following corollary is a characterization for minimal prime filters belonging to a filter.
\begin{theorem}\label{mincor}
Let $\mathfrak{A}$ be a residuated lattice, $F$ be a filter and $P$ be a prime filter containing $F$. The following assertions are equivalent:
\begin{enumerate}
  \item  [$(1)$ \namedlabel{mincor1}{$(1)$}] $P$ is an $F$-minimal prime filter;
  \item  [$(2)$ \namedlabel{mincor2}{$(2)$}] $P=D_{F}(P)$;
  \item  [$(3)$ \namedlabel{mincor3}{$(3)$}] for any $x\in A$, $P$ contains precisely one of $x$ or $(F:x)$.
\end{enumerate}
\end{theorem}
\begin{proof}
\item [] \ref{mincor1}$\Rightarrow$ \ref{mincor2}: It follows by Proposition \ref{minlem1}\ref{minlem11}.
\item [] \ref{mincor2}$\Rightarrow$ \ref{mincor3}: It is a direct consequence of Proposition \ref{lemdivi}\ref{lemdivi1}.
\item [] \ref{mincor3}$\Rightarrow$ \ref{mincor1}: Let $Q$ be a prime filter containing $F$ such that $Q\subseteq P$. Consider $x\in P$. So $(F:x)\nsubseteq P$ and it implies that $x\in D_{F}(P)\subseteq D_{F}(Q)\subseteq Q$ and this shows that $P=Q$.
\end{proof}
\begin{corollary}
Let $\mathfrak{A}$ be a residuated lattice and $F$ be a filter of $\mathfrak{A}$. For any two distinct $F$-minimal prime filters $\mathfrak{m}_1$ and $\mathfrak{m}_2$ we have $\mathfrak{m}_1\vee^{\omega_F} \mathfrak{m}_2=A$.
\end{corollary}
\begin{proof}
Let $\mathfrak{m}_1$ and $\mathfrak{m}_2$ be two distinct $F$-minimal prime filters of $\mathfrak{A}$. Let $a\in \mathfrak{m}_1\setminus \mathfrak{m}_2$ and $b\in \mathfrak{m}_2\setminus\mathfrak{m}_1$. By Proposition \ref{mincor}\ref{mincor3} follows that $(F:a)\nsubseteq \mathfrak{m}_1$ and so there exists some $c\in (F:a)\setminus \mathfrak{m}_1$. So $a\vee (b\vee c)\in F$, $(F:b\vee c)\subseteq \mathfrak{m}_1$ and $(F:a)\subseteq \mathfrak{m}_2$. Hence, by Corollary \ref{lastcor} we have $A=(F:a)\vee^{\omega_{F}}(F:b\vee c)\subseteq \mathfrak{m}_1\vee^{\omega_{F}} \mathfrak{m}_2$.
\end{proof}
\begin{proposition}\label{minosub}
Let $\mathfrak{A}$ be a residuated lattice, $F$ be a filter and $\mathscr{C}$ be a $\vee$-closed subset of $\mathfrak{A}$. If $\mathfrak{m}$ is an $\omega_{F}(\mathscr{C})$-minimal prime filter, then $\mathfrak{m}\cap \mathscr{C}=\emptyset$.
\end{proposition}
\begin{proof}
Let $\mathfrak{m}$ be an $\omega_{F}(\mathscr{C})$-minimal prime filter and $x\in \mathfrak{m}\cap \mathscr{C}$.  By Theorem \ref{mincor}, we have $x\in D_{\omega_{F}(\mathscr{C})}(\mathfrak{m})$ and it implies that $x\vee y\in \omega_{F}(\mathscr{C})$ for some $y\notin \mathfrak{m}$. So there exists $c\in \mathscr{C}$ such that $x\vee y\in (F:c)$ and it follows that $y\in (F:x\vee c)\subseteq \omega_{F}(\mathscr{C})\subseteq \mathfrak{m}$. It leads us to a contradiction.
\end{proof}

\begin{corollary}
Let $\mathfrak{A}$ be a residuated lattice, $F$ be a filter and $P$ be a prime filter. Then any $D_{F}(P)$-minimal prime filter is contained in $P$.
\end{corollary}
\begin{proof}
By takin $\mathscr{C}=P^{c}$ it follows by Proposition \ref{minosub}.
\end{proof}
\begin{proposition}\label{omegaminpro}
Let $\mathfrak{A}$ be a residuated lattice, $F$ be a filter and $\mathscr{C}$ be a $\vee$-closed subset. We have
\[Min_{\omega_{F}(\mathscr{C})}(\mathfrak{A})=\{\mathfrak{m}|\mathfrak{m}\in Min_{F}(\mathfrak{A}),~\mathfrak{m}\cap \mathscr{C}=\emptyset\}.\]
\end{proposition}
\begin{proof}
Set $\mu=\{\mathfrak{m}\in Min_{F}(\mathfrak{A})|\mathfrak{m}\cap \mathscr{C}=\emptyset\}$. If $\mathfrak{m}\in \mu$, then by Proposition \ref{1propdiv} and Proposition \ref{minlem1} follows that
\[\omega_{F}(\mathscr{C})\subseteq D_{F}(\mathfrak{m})(=\mathfrak{m})\subseteq D_{\omega_{F}(\mathscr{C})}(\mathfrak{m})\subseteq \mathfrak{m}.\]

It follows that $\mathfrak{m}=D_{\omega_{F}(\mathscr{C})}(\mathfrak{m})$ and so $\mathfrak{m}\in Min_{\omega_{F}(\mathscr{C})}(\mathfrak{A})$.

Conversely, let $\mathfrak{m}\in Min_{\omega_{F}(\mathscr{C})}(\mathfrak{A})$. By Proposition \ref{1propdiv}\ref{1propdiv3} follows that $F\subseteq \mathfrak{m}$ and by Proposition \ref{minosub} follows that $\mathfrak{m}\cap \mathscr{C}=\emptyset$. Suppose that $\mathfrak{w}$ is a prime filter containing $F$ such that $\mathfrak{w}\cap \mathscr{C}=\emptyset$ and $\mathfrak{w}\subseteq \mathfrak{m}$. Applying Proposition \ref{lemdivi}, it shows that $\omega_{F}(\mathscr{C})\subseteq D_{F}(\mathfrak{w})\subseteq \mathfrak{w}$. Therefore, $\mathfrak{w}$ is a prime filter containing $\omega_{F}(\mathscr{C})$ and so $\mathfrak{w}=\mathfrak{m}$. It shows that $\mathfrak{m}$ is an $F$-minimal prime filter and so $\mathfrak{m}\in \mu$.
\end{proof}

The following corollary is an immediate consequence of Proposition \ref{omegaminpro}.
\begin{corollary}\label{lastcoce}
Let $\mathfrak{A}$ be a residuated lattice, $F$ be a filter and $P$ be a prime filter. We have
\[Min_{D_{F}(P)}(\mathfrak{A})=\{\mathfrak{m}|\mathfrak{m}\in Min_{F}(\mathfrak{A}),~\mathfrak{m}\subseteq P\}.\]
\end{corollary}
\begin{proof}
By takin $\mathscr{C}=P^{c}$ it follows by Proposition \ref{omegaminpro}.
\end{proof}
\begin{corollary}\label{ofmpr}
Let $\mathfrak{A}$ be a residuated lattice, $F$ be a filter and $\mathscr{C}$ be a $\vee$-closed subset. We have
\[\omega_{F}(\mathscr{C})=\bigcap \{\mathfrak{m}|\mathfrak{m}\in Min_{F}(\mathfrak{A}),~\mathfrak{m}\cap \mathscr{C}=\emptyset\}.\]
\end{corollary}
\begin{proof}
It follows by Corollary \ref{mininters} and Proposition \ref{omegaminpro}.
\end{proof}
\begin{corollary}\label{dfmpr}
Let $\mathfrak{A}$ be a residuated lattice, $F$ be a filter and $P$ be a prime filter. We have
\[D_{F}(P)=\bigcap \{\mathfrak{m}|\mathfrak{m}\in Min_{F}(\mathfrak{A}),~\mathfrak{m}\subseteq P\}.\]
\end{corollary}
\begin{proof}
By takin $\mathscr{C}=P^{c}$ it follows by Corollary \ref{ofmpr}.
\end{proof}
\section{$n$-normal residuated lattices}\label{sec4}

In this section we introduce and study the notions of normal and $n$-normal residuated lattices which are inspired by the study of normal lattices \citep{cor} and $n$-normal lattices \citep{cor74}. We characterize these classes of residuated lattices in terms of $\omega$-filters.


\begin{lemma}\label{npri}
Let $\mathfrak{A}$ be a residuated lattice and $F$ be a filter of $\mathfrak{A}$. For a given integer $n\geq 2$, the following assertions are equivalent:
\begin{enumerate}
\item  [$(1)$ \namedlabel{npri1}{$(1)$}] For any filters $F_1,\cdots,F_n$ such that $F_i\cap F_j=F$ for any $i\neq j$, there exists $k$ such that $F_k=F$;
\item  [$(2)$ \namedlabel{npri2}{$(2)$}] for any filters $F_1,\cdots,F_n$ such that $F_i\cap F_j\subseteq F$ for any $i\neq j$, there exists $k$ such that $F_k\subseteq F$;
\item  [$(3)$ \namedlabel{npri3}{$(3)$}] for any $x_1,\cdots,x_n\in A$ which are ``\textit{pairwise}" in $F$, i.e. $x_i\vee x_j\in F$ for any $i\neq j$, there exists $k$ such that $x_k\in F$;
\item  [$(4)$ \namedlabel{npri4}{$(4)$}] $F$ is the intersection of at most $n-1$ distinct prime filters.
\end{enumerate}
\end{lemma}
\begin{proof}
\item [] \ref{npri1}$\Rightarrow$\ref{npri2}: Let $F_1,\cdots,F_n$ be $n$ filters such that $F_i\cap F_j\subseteq F$ for any $i\neq j$. Consider $F_1\veebar F,\cdots,F_n\veebar F$. For any $i\neq j$ we have $(F_i\veebar f)\cap (F_j\veebar f)=(F_i\cap F_j)\veebar F=F$. So there exists $k$ such that $F_k\veebar F=F$. So $F_k\subseteq F$ .
\item [] \ref{npri2}$\Rightarrow$\ref{npri3}: Let $x_1,\cdots,x_n\in A$ which are pairwise in $F$. Consider $\mathscr{F}(F,x_1),\cdots,\mathscr{F}(F,x_n)$. Let $z\in \mathscr{F}(F,x_i)\cap \mathscr{F}(F,x_j)$. So by \textsc{Remark} \ref{genfilprop}\ref{genfilprop1} there exist $f_i,f_j\in P$ and integers $n_i,n_j$ such that $z\geq (f_i\odot x_{i}^{n_i})\vee (f_j\odot x_{j}^{n_j})$. By \ref{res2} we deduce $z\geq (f_i\vee f_j)\odot (f_i\vee x_{j}^{n_j})\odot (x_{i}^{n_i}\vee f_j)\odot (x_{i}\vee x_{j})^{n_in_j}$ and so $z\in F$. Hence $\mathscr{F}(F,x_i)\cap \mathscr{F}(F,x_j)\subseteq F$. So there exists $k$ such that $x_k\in \mathscr{F}(F,x_k)\subseteq F$.
\item [] \ref{npri3}$\Rightarrow$\ref{npri4}: If $n=2$, then $F$ is a prime filter and so \ref{npri4} is obviously holds. Let $m<n$ be the largest integer such that \ref{npri3} does not hold for $F$. So there exist $a_{1},\cdots,a_{m}\in A$ pairwise in $F$, yet $a_{1},\cdots,a_{m}\notin F$. We show that $(F:a_i)$ is a prime filter for any $1\leq i\leq m$. Consider $1\leq i\leq m$. By Proposition \ref{1fxpro}\ref{1fxpro2} follows that $(F:a_i)$ is a proper filter. Let $b\vee c\in (F:a_i)$. Consider the set of $m+1$ elements $\{a_1,\cdots,a_{i-1},b\vee a_i,c\vee a_i,a_{i+1},\cdots,a_{m}\}$. This set is pairwise in $F$ and so $b\vee a_i\in F$ or $c\vee a_i\in F$. It implies that $b\in (F:a_i)$ or $c\in (F:a_{i})$, hence $(F:a_{i})$ is a prime filter.

   Obviously, we have $F\subseteq \cap_{i=1}^{m} (F:a_i)$. If $w\in \cap_{i=1}^{m} (F:a_i)$, then $a_1,\cdots,a_{m},w$ are pairwise in $F$ and so $w\in F$. It shows that $F=\cap_{i=1}^{m} (F:a_i)$ is the intersection of $m<n$ prime filters.
\item [] \ref{npri4}$\Rightarrow$\ref{npri1}: Let $P_1\cdots,P_m$ ($1\leq m<n$) are distinct prime filters such that   $F=\cap_{i=1}^{m}P_i$. Let $F_1,\cdots,F_n$ be $n$ filters of $\mathfrak{A}$ such that $F_i\cap F_j=F$ for any $i\neq j$. Let $F_j\nsubseteq P_{i_j}$ for $1\leq j,i_j\leq m$. So by Pigeonhole principle there exists some $m<k$ such that $F_k\subseteq \cap_{i=1}^{m}P_i=F$. Hence, $F_k=F$.
\end{proof}
\begin{definition}
Let $\mathfrak{A}$ be a residuated lattice and $P$ be a proper filter of $\mathfrak{A}$. $P$ is called \textit{$n$-prime} if it satisfies any of the equivalent assertions of Lemma \ref{npri}.
\end{definition}
\begin{definition}\label{conordef}
Let $\mathfrak{A}$ be a residuated lattice and $F$ be a filter of $\mathfrak{A}$. $\mathfrak{A}$ is called \textit{$n$-normal with respect to $F$} if any prime filter containing $F$ contains at most $n$ $F$-minimal prime filter. $\mathfrak{A}$ is called \textit{normal with respect to $F$} if it is $1$-normal with respect to $F$. $\mathfrak{A}$ is called \textit{normal} if it is normal with respect to $\{1\}$.
\end{definition}
\begin{lemma}\label{comincor}
Let $\mathfrak{A}$ be a residuated lattice and $F$ be a filter of $\mathfrak{A}$ and $n\geq 2$. If $\mathfrak{m}_1,\cdots,\mathfrak{m}_n$ are distinct $F$-minimal prime filters. Then there exist $a_1,\cdots,a_n\in A$ which are pairwise in $F$ and $a_i\notin \mathfrak{m}_i$ for any $1\leq i\leq n$. Moreover, if we set $b_i=\odot_{j=1\atop j\neq i}^{n} a_j$, the following assertions hold:
\begin{enumerate}
  \item [$(1)$ \namedlabel{comincor1}{$(1)$}] $b_i\in \mathfrak{m}_i$ for $1\leq i\leq n$;
  \item [$(2)$ \namedlabel{comincor2}{$(2)$}] $\vee_{i=1}^{n}b_i\in F$;
  \item [$(3)$ \namedlabel{comincor3}{$(3)$}] $(F:\vee_{j=1\atop j\neq i}^{n}b_j)\subseteq \mathfrak{m}_i$ for any $1\leq i\leq n$.
\end{enumerate}
\end{lemma}
\begin{proof}
Let $n=2$. So there exist $x_1\in \mathfrak{m}_2-\mathfrak{m}_1$ and $x_2\in \mathfrak{m}_1-\mathfrak{m}_2$. Applying Theorem \ref{mincor} there exists $y_2\notin \mathfrak{m}_1$ such that $x_2\vee y_2\in F$. It follows that $a_1=x_1\vee y_2$ and $a_2=x_2$ establish the result. Suppose that the result holds for $n-1$ and $\mathfrak{m}_1,\cdots,\mathfrak{m}_n$ are distinct $F$-minimal prime filters. Let $x_1,\cdots,x_{n-1}$ are pairwise in $F$ and $x_i\notin \mathfrak{m}_i$ for $1\leq i\leq n-1$. Consider $y_i\in \mathfrak{m}_n-\mathfrak{m}_i$ for $1\leq i\leq n$. Let $y=\odot_{i=1}^{n}y_i$, hence $y\in \mathfrak{m}_n-(\bigcup_{i=1}^{n-1} \mathfrak{m}_i)$. By Theorem \ref{mincor} there exists $z\in \mathfrak{m}_n$ such that $y\vee z\in F$. It follows that $a_i=x_i\vee y$ ($i=1,\cdots,n-1$) and $a_n=z$ are the required elements.
\item [\ref{comincor1}:] Let $1\leq i\leq n$. For any $1\leq j\neq i\leq n$ we have $a_i\vee a_j\in F$ and $a_i\notin \mathfrak{m}_i$. So $a_j\in \mathfrak{m}_{i}$ and it shows that $b_i\in \mathfrak{m}_{i}$.
\item [\ref{comincor2}:] By a simple induction on \ref{res2} follows that $\vee_{i=1}^{n}b_i\geq \odot_{j_i\in \{1,\cdots,i-1,i+1,\cdots,n\}} (\vee_{i=1}^{n} a_{j_i})$. Since $a_1,\cdots,a_{n}$ are pairwise in $F$ so the result holds.
\item [\ref{comincor3}:] For any $j\neq i$ we have $b_j\leq a_i$. So if $\vee_{j=1\atop j\neq i}^{n}b_j\in \mathfrak{m}_i$, then $a_i\in \mathfrak{m}_i$; a contradiction. Hence, $\vee_{j=1\atop j\neq i}^{n}b_j\notin \mathfrak{m}_i$ and so the result establishes by Theorem \ref{mincor}\ref{mincor3}.
\end{proof}
\begin{proposition}\label{conolem}
Let $\mathfrak{A}$ be a residuated lattice and $F$ be a filter of $\mathfrak{A}$. The following assertions are equivalent:
\begin{enumerate}
\item  [$(1)$ \namedlabel{conolem1}{$(1)$}] For any $n+1$ distinct $F$-minimal prime filters $\mathfrak{m}_{0},\cdots,\mathfrak{m}_{n}$,
    \[\veebar_{i=1}^{n}\mathfrak{m}_{i}=A;\]
\item  [$(2)$ \namedlabel{conolem2}{$(2)$}] $\mathfrak{A}$ is $n$-normal with respect to $F$;
\item  [$(3)$ \namedlabel{conolem3}{$(3)$}] for any prime filter $P$ containing $F$, $D_{F}(P)$ is an $(n+1)$-prime filter;
\item  [$(4)$ \namedlabel{conolem4}{$(4)$}] for any $x_0,\cdots,x_n\in A$ which are pairwise in $F$,
\[\veebar_{i=0}^{n}(F:x_{i})=A;\]
\item  [$(5)$ \namedlabel{conolem5}{$(5)$}] for any $x_0,\cdots,x_n\in A$ which are pairwise in $F$, there exists $a_i\in (F:x_i)$ for any $0\leq i\leq n$ such that $\odot_{i=0}^{n} a_i=0$;
\item  [$(6)$ \namedlabel{conolem6}{$(6)$}] for any $x_0,\cdots,x_n\in A$, $(F:\vee_{i=0}^{n} x_{i})=\veebar_{i=0}^{n}(F:\vee_{j=0\atop i\neq j}^{n}x_{j})$;
\item  [$(7)$ \namedlabel{conolem7}{$(7)$}] for any $x_0,\cdots,x_n\in A$, $\vee_{i=0}^{n} x_{i}\in F$ implies $\veebar_{i=0}^{n}(F:\vee_{j=0\atop i\neq j}^{n}x_{j})=A$;
\end{enumerate}
\end{proposition}
\begin{proof}
\ref{conolem1}$\Rightarrow$\ref{conolem2} is trivial and \ref{conolem2}$\Rightarrow$\ref{conolem3}
  is a direct consequence of Corollary \ref{dfmpr} and Lemma \ref{npri}\ref{npri4}.
\item [] \ref{conolem3}$\Rightarrow$\ref{conolem4}: Let $x_1,\cdots,x_n\in A$ are pairwise in $F$. If $\veebar_{i=1}^{n}(F:x_{i})\neq A$, then there exists a prime filter $P$ containing $\veebar_{i=1}^{n}(F:x_{i})$. So by Theorem \ref{mincor}\ref{mincor3} follows that $x_1,\cdots,x_n\notin D_{F}(P)$ and it leads us to a contradiction following by Lemma \ref{npri}\ref{npri3}.
\item [] \ref{conolem4}$\Rightarrow$\ref{conolem5}: Let $x_1,\cdots,x_n\in A$ are pairwise in $F$. Hence $\veebar_{i=1}^{n}(F:x_{i})=A$ and so by \textsc{Remark} \ref{genfilprop}\ref{genfilprop1} there exist $a_i\in (F:x_i)$ for any $1\leq i\leq n$ such that $\odot_{i=1}^{n} a_i=0$.
\item [] \ref{conolem5}$\Rightarrow$\ref{conolem6}: Let $a\in (F:\vee_{i=1}^{n} x_{i})$. Let $b_{i}=a\vee(\vee_{j\neq i} x_j)$. Obviously, $b_1,\cdots,b_n$ are pairwise in $F$. So there exits $a_i\in (F:b_i)$ for any $1\leq i\leq n$ such that $\odot_{i=1}^{n} a_i=0$. By \ref{res2} follows that $a=a\vee \odot_{i=1}^{n} a_i\geq \odot_{i=1}^{n} (a\vee a_i)$. On the other hand, $a\vee a_i\in (F:\vee_{j=0\atop i\neq j}^{n}x_{j})$ for any $1\leq i\leq n$. By \textsc{Remark} \ref{genfilprop}\ref{genfilprop1} follows that $a\in \veebar_{i=0}^{n}(F:\vee_{j=0\atop i\neq j}^{n}x_{j})$. The other inclusion follows by \textsc{Remark} \ref{4fxpro}\ref{4fxpro4}.
\item [] \ref{conolem6}$\Rightarrow$\ref{conolem7}: It is trivial.
\item [] \ref{conolem7}$\Rightarrow$\ref{conolem1}: Let $\mathfrak{m}_{0},\cdots,\mathfrak{m}_{n}$ be distinct $F$-minimal prime filters. By Lemma \ref{comincor}\ref{comincor2} follows that $\vee_{i=1}^{n}b_i\in F$ and so $\veebar_{i=0}^{n}(F:\vee_{j=0\atop i\neq j}^{n}b_{j})=A$. Also, by Lemma \ref{comincor}\ref{comincor3}, $(F:\vee_{j=0\atop i\neq j}^{n}b_{j})\subseteq \mathfrak{m}_{i}$ for any $0\leq i\leq n$. so $\veebar_{i=1}^{n}\mathfrak{m}_{i}=A$.
\end{proof}
\begin{corollary}\label{noco}
Let $\mathfrak{A}$ be a residuated lattice. The following assertions are equivalent:
\begin{enumerate}
\item  [$(1)$ \namedlabel{noco1}{$(1)$}] Any two distinct minimal prime filters are comaximal;
\item  [$(2)$ \namedlabel{noco2}{$(2)$}] $\mathfrak{A}$ is normal;
\item  [$(3)$ \namedlabel{noco3}{$(3)$}] for any prime filter $P$, $D(P)$ is prime;
\item  [$(4)$ \namedlabel{noco4}{$(4)$}] for any $x,y\in A$, $x\vee y=1$ implies $x^{\perp}\veebar y^{\perp}=A$;
\item  [$(5)$ \namedlabel{noco5}{$(5)$}] for any $x,y\in A$, $x\vee y=1$ implies that there exist $u\in x^{\perp}$ and $v\in y^{\perp}$ such that $u\odot v=0$;
\item  [$(6)$ \namedlabel{noco6}{$(6)$}] for any $x,y\in A$, $(x\vee y)^{\perp}=x^{\perp}\veebar y^{\perp}$;
\item  [$(7)$ \namedlabel{noco7}{$(7)$}] for any $x,y\in A$, $(x\vee y)^{\perp}=A$ implies $x^{\perp}\veebar y^{\perp}=A$;
\end{enumerate}
\end{corollary}
\begin{proof}
It follows by taking $F=\{1\}$ in Proposition \ref{conolem}.
\end{proof}

\begin{proposition}\label{osublatfil}
Let $\mathfrak{A}$ be a residuated lattice. The following assertions are equivalent:
\begin{enumerate}
\item  [$(1)$ \namedlabel{osublatfil1}{$(1)$}] for any $F,G\in \Omega(\mathfrak{A})$, $F\vee^{\omega} G=A$ implies $F\veebar G=A$;
\item  [$(2)$ \namedlabel{osublatfil2}{$(2)$}] $\mathfrak{A}$ is normal;
\item  [$(3)$ \namedlabel{osublatfil3}{$(3)$}] for any $\mathcal{F}\subseteq \Omega(\mathfrak{A})$ we have $\veebar \mathcal{F}\in \Omega(\mathfrak{A})$;
\item  [$(4)$ \namedlabel{osublatfil4}{$(4)$}] $\Omega(\mathfrak{A})$ is a sublattice of $\mathscr{F}(\mathfrak{A})$;
\item  [$(5)$ \namedlabel{osublatfil5}{$(5)$}] $\gamma(\mathfrak{A})$ is a sublattice of $\mathscr{F}(\mathfrak{A})$;
\end{enumerate}
\end{proposition}
\begin{proof}
\item  []\ref{osublatfil1}$\Rightarrow$\ref{osublatfil2}: Let $x\vee y=1$ for some $x,y\in A$. Since $\gamma(\mathfrak{A})$ is a sublattice of $\Omega(\mathfrak{A})$ so we have $x^{\perp}\vee^{\omega} y^{\perp}=x^{\perp}\vee^{\Gamma} y^{\perp}=(x\vee y)^{\perp}=A$. Thus $\mathfrak{A}$ is normal due to Corollary \ref{noco}\ref{noco4}.
\item  []\ref{osublatfil2}$\Rightarrow$\ref{osublatfil3}: Let $\{F_i\}_{i\in I}$ be a family of $\omega$-filters and let for any $i\in I$, $I_i$ be a lattice ideal such that $F_i=\omega(I_i)$. By Proposition \ref{1propdiv}\ref{1propdiv4} follows that for any $i\in I$ we have $F_i\subseteq \omega(\curlyvee_{i\in I}I_i)$ and it states that $\veebar_{i\in I} F_i\subseteq \omega(\curlyvee_{i\in I}I_i)$ since $\omega(\curlyvee_{i\in I}I_i)$ is a filter. Let $a\in \omega(\curlyvee_{i\in I}I_i)$. Hence, there exists $x\in \curlyvee_{i\in I}I_i$ such that $a\in x^{\perp}$. It implies that $x\leq x_{i_1}\vee\cdots\vee x_{i_n}$ for some integer $n$ and $x_{i_j}\in I_{i_j}$. So by Proposition \ref{4fxpro}\ref{4fxpro1} and Corollary \ref{noco}\ref{noco6} we have the following sequence of formulas:
     \[
     \begin{array}{ll}
        x^{\perp}&\subseteq (x_{i_1}\vee\cdots\vee x_{i_n})^{\perp} \\
        &=x^{\perp}_{i_1}\veebar\cdots\veebar x^{\perp}_{i_n} \\
        &\subseteq  F_{i_1}\veebar\cdots\veebar F_{i_n}\\
        &\subseteq \veebar_{i\in I} F_i.
     \end{array}
     \]
It shows that $\veebar_{i\in I} F_i=\omega(\curlyvee_{i\in I}I_i)$.
\item  []\ref{osublatfil3}$\Rightarrow$\ref{osublatfil4}: Let $F,G\in \Omega(\mathfrak{A})$. By Proposition \ref{dislattoff} we have $F\veebar G\subseteq F\vee^{\omega} G$ and by \ref{osublatfil3} we have $F\vee^{\omega} G\subseteq F\veebar G$. It holds the result.
\item  []\ref{osublatfil4}$\Rightarrow$\ref{osublatfil5}: It is trivial.
\item  []\ref{osublatfil5}$\Rightarrow$\ref{osublatfil1}: Let $F,G\in \Omega(\mathfrak{A})$ such that $F\vee^{\omega} G=A$. Since $\omega(I_F\curlyvee I_G)=A$, so by Proposition \ref{1propdiv}\ref{1propdiv6} follows that $1\in I_F\curlyvee I_G$ and it states that $f\vee g=1$ for some $f\in I_F$ and $g\in I_G$. Hence, $A=(f\vee g)^{\perp}=f^{\perp}\vee^{\Gamma} g^{\perp}=f^{\perp}\veebar g^{\perp}\subseteq F\veebar G$.
\end{proof}

In light of above corollary, we obtain the existence of the greatest $\omega$-filters contained in a given filter of a normal residuated lattice.
\begin{proposition}
Let $\mathfrak{A}$ be a normal residuated lattice. Then for any filter $F$ there exists a largest $\omega$-filter contained in $F$.
\end{proposition}
\begin{proof}
Let $F$ be a filter and $\mathcal{F}$ be the family of $\omega$-filters of $\mathfrak{A}$ contained in $F$. By Corollary \ref{osublatfil}\ref{osublatfil3}, $\veebar \mathcal{F}$ is an $\omega$-filter and obviously it is the largest $\omega$-filter contained in $F$.
\end{proof}
\begin{definition}
Let $\mathfrak{A}$ be a residuated lattice. For any filter $F$ of $\mathfrak{A}$ we set
\[\sigma(F)=\{a\in A|a^{\perp}\veebar F=A\}.\]
\end{definition}
\begin{proposition}\label{sigmaomegafil}
Let $\mathfrak{A}$ be a residuated lattice and $F$ be a filter of $\mathfrak{A}$. Then $\sigma(F)$ is an $\omega$-filter contained in $F$.
\end{proposition}
\begin{proof}
Let $I_F=\{a\in A|a^{\perp\perp}\veebar F=A\}$. Let $a,b\in I_F$. By Proposition \ref{4fxpro}\ref{4fxpro3} and distributivity of $\mathscr{F}(\mathfrak{A})$ follows that $(a\vee b)^{\perp\perp}\veebar F=(a^{\perp\perp}\cap b^{\perp\perp})\veebar F=(a^{\perp\perp}\cap F)\veebar (a^{\perp\perp}\cap F)=A$. So $I_F$ is a $\vee$-closed subset of $\mathfrak{A}$. Now, let $a\leq b$ and $b\in I_F$. By Proposition \ref{1fxpro}\ref{1fxpro1} and Proposition \ref{4fxpro}\ref{4fxpro1} follows that $b^{\perp\perp}\subseteq a^{\perp\perp}$ and it implies that $a\in I_F$. Thus $I_F$ is an ideal and so $\omega(I_F)$ is a filter due to Proposition \ref{2propdiv}.

Let $a\in \omega(I_F)$. So $a\in x^{\perp}$ for some $x\in I_F$. By Proposition \ref{1fxpro}\ref{1fxpro1} follows that $A=x^{\perp\perp}\veebar F\subseteq a^{\perp}\veebar F$ and so $a\in \sigma(F)$. Conversely, let $a\in \sigma(F)$. So $x\odot f=0$ for some $x\in a^{\perp}$ and $f\in F$. Since $x\in x^{\perp\perp}$ we obtain that $x\in I_F$. On the other hand, $a\in a^{\perp\perp}\subseteq x^{\perp}$ so $a\in \omega(I_F)$. Hence, $\sigma(F)=\omega(I_F)$ and it proves that $\sigma(F)$ is an $\omega$-filter.

At the end, for any $a\in \sigma(F)$ there exist $x\in a^{\perp}$ and $f\in F$ such that $x\odot f=0$. By \ref{res1} we have
\[f=f\odot 1=f\odot (a\vee x)=(f\odot a)\vee (f\odot x)=f\odot a.\]
So $f\leq a$ and it shows that $a\in F$. It holds the result.
\end{proof}

In the following we characterize the greatest $\omega$-filter of a normal residuated lattice $\mathfrak{A}$ contained in a given filter.
\begin{theorem}
Let $\mathfrak{A}$ be a normal residuated lattice and $F$ be a filter of $\mathfrak{A}$. Then $\sigma(F)$ is the greatest $\omega$-filter of $\mathfrak{A}$ contained in $F$.
\end{theorem}
\begin{proof}
By Proposition \ref{sigmaomegafil} follows that $\sigma(F)$ is an $\omega$-filter contained in $F$. Let $G$ be an $\omega$-filter such that $G\subseteq F$. Thus for $a\in G$ there exists $x\in I_G$ such that $a\vee x=1$. By Corollary \ref{noco}\ref{noco4} follows that $A=a^{\perp}\veebar x^{\perp}\subseteq a^{\perp}\veebar G\subseteq a^{\perp}\veebar F$. It shows that $a\in \sigma(F)$ and so the result holds.
\end{proof}

\end{document}